\documentclass[12pt,twoside]{article}

\usepackage{a4}
\usepackage{amssymb,amsmath,amsthm,latexsym, mathrsfs}
\usepackage{amsfonts}
\usepackage{amsfonts}
\usepackage{graphicx}
\usepackage{amsmath, amsfonts}
\usepackage{amssymb, graphicx}
\usepackage{amscd}
\usepackage{textcomp}
\usepackage{palatino}
\usepackage{xcolor}
\usepackage{array}
\usepackage{multirow}
\usepackage{diagbox}
\usepackage[colorlinks=true,linkcolor=red,citecolor=red]{hyperref}

\newtheorem{theorem}{Theorem}[section]

\newtheorem{conjecture}{Conjecture}
\newtheorem{corollary}[theorem] {Corollary}
\newtheorem{definition}[theorem]{Definition}

\newtheorem{lemma} [theorem]{Lemma}

\usepackage{indentfirst}
\allowdisplaybreaks
\vspace{5cm}

\begin{document}
	\label{'ubf'}  
	\setcounter{page}{1}
	\markboth
	{\hspace*{-9mm} \centerline{\footnotesize $q$-analogue of the generalized Stieltjes constants}}
	{\centerline{\footnotesize Tapas Chatterjee and Sonam Garg } \hspace*{-9mm}}
	\vspace*{-2cm}
	
\begin{center}
	{{\textbf{Linear independence of $q$-analogue of the generalized Stieltjes constants over number fields }}\\
	\vspace{.2cm}
	\medskip
	{\sc Tapas Chatterjee \footnote{Research of the first author is partly supported by the core research grant CRG/2023/000804 of the Science and Engineering Research Board of DST, Government of India.}}\\
	{\footnotesize  Department of Mathematics,}\\
	{\footnotesize Indian Institute of Technology Ropar, Punjab, India.}\\
	{\footnotesize e-mail: {\it tapasc@iitrpr.ac.in}}\\
	\medskip
	{\sc Sonam Garg\footnote{Research of the second author is supported by the University Grants Commission (UGC), India under File No.: 972/(CSIR-UGC NET JUNE 2018).}}\\
	{\footnotesize Department of Mathematics, }\\
	{\footnotesize Indian Institute of Technology Ropar, Punjab, India.}\\
	{\footnotesize e-mail: {\it 2018maz0009@iitrpr.ac.in}}
	\medskip}
\end{center}
	
\thispagestyle{empty} 
\vspace{-.4cm}
\hrulefill
	
\begin{abstract}  
	{\footnotesize }
	In this article, we aim to extend the research conducted by Chatterjee and Garg in 2024, particularly focusing on the $q$-analogue of the generalized Stieltjes constants. These constants constitute the coefficients in the Laurent series expansion of a $q$-analogue of the Hurwitz zeta function around $s=1$. Chatterjee and Garg previously established arithmetic results related to $\gamma_0(q,x)$, for $q>1$ and $0 < x <1$ over the field of rational numbers. Here, we broaden their findings to encompass number fields $\mathbb{F}$ in two scenarios: firstly, when $\mathbb{F}$ is linearly disjoint from the cyclotomic field $\mathbb{Q}(\zeta_b)$, and secondly, when $\mathbb{F}$ has non-trivial intersection with $\mathbb{Q}(\zeta_b)$, with $b \geq 3$ being any positive integer.
\end{abstract}
\hrulefill
		
{\small \textbf{Key words and phrases}: }Baker's theory, Bernoulli polynomials, Cyclotomic fields, Number fields, Okada's criterion, $q$-Hurwitz zeta function. \\
		
{\bf{Mathematics Subject Classification 2020:}} 05A30, 11J81, 11J72, 11M35, 11B68.
	
\vspace{-.37cm}
	
\section{\bf Introduction}
	
	The Hurwitz zeta function was introduced by Adolf Hurwitz in 1882, as one of the generalizations of the Riemann zeta function. It is defined as follows:
\begin{align*}
	\zeta(s,a) = \sum_{n=0}^{\infty}\frac{1}{(n+a)^s},
\end{align*}
	where $s$ is a complex number with real part greater than $1$, and $0 < a \leq 1$. Much like the Riemann zeta function, it can be extended analytically to the entire complex plane, except at $s=1$, where it has a simple pole with a residue of one. This characteristic is evident in its Laurent series expansion:
	$$\zeta(s,a) = \frac{1}{s-1} + \displaystyle \sum_{n=0}^{\infty}\frac{(-1)^n}{n!}\gamma_n(a)(s-1)^n,$$ 
	where $\gamma_n(a)$ are the generalized Stieltjes constants. It plays a crucial role in studying prime number distribution and modular forms. Noteworthy properties of the function include functional equations, integral representations, and connections to special functions like the gamma function and polylogarithms. Since its inception, mathematicians have extensively explored the Hurwitz zeta function, establishing its significance in various mathematical disciplines. 
	
	In a different direction, there has been a growing interest among mathematicians in exploring alternative generalizations of classical functions known as $q$-analogues. A $q$-analogue represents a modified version of a function or mathematical structure, typically involving a parameter $q$. As $q$ approaches $1$, the $q$-analogue converges to the original function or structure. These $q$-analogues are not unique, with several variations explored for each classical function. A notable classical function in this exploration is the Riemann zeta function. In 2003, Kurokawa and Wakayama \cite{KW} directed their focus towards a specific $q$-analogue of the Riemann zeta function, denoted as $\zeta_q(s)$:
\begin{align*} 
	\zeta_q(s) = \sum_{n=1}^\infty\frac{q^n}{[n]_q^s}, ~~~~~~~~~ \Re(s) > 1.
\end{align*}

	Further, Chatterjee and Garg as outlined in ref.\,\cite{TSG, TSG1, TSG2} investigated this variant, its generalization, and its algebraic identities. In particular, they indulge in the coefficients in the Laurent series expansion of these variants. Also, they explored the transcendence and linear independence results associated with these coefficients. In addition to the above, Kurokawa and Wakayama addressed a $q$-analogue of the Hurwitz zeta function in ref.\,\cite{KW}. Specifically, they investigated the following $q$-analogue of the Hurwitz zeta function, denoted as $\zeta_q(s,x)$, for the case when $q>1$:
\begin{align} \label{E2}
	\zeta_q(s,x) = \sum_{n=0}^\infty\frac{q^{n+x}}{[n+x]_q^s}, ~~~~ \Re(s)>1,
\end{align}
	where $x \notin \mathbb{Z}_{\leq 0}$. Subsequently, Chatterjee and Garg in ref.\,\cite{TSG3}, further examined this $q$-variant of the Hurwitz zeta function and gave the closed-form expression for the first two coefficients in its Laurent series expansion around $s=1$. In this context, they gave the following theorem:
	
\begin{theorem}\label{T1}
	The $q$-analogue of the Hurwitz zeta function defined in Equation (\ref{E2}) is meromorphic for $s\in \mathbb{C}$ and its Laurent series expansion around $s=1$ is given by:
\begin{align*}
	\zeta_q(s,x) = \frac{q-1}{\log q}.\frac{1}{s-1} + \gamma_0(q,x) + \gamma_1(q,x)(s-1) + \gamma_2(q,x)(s-1)^2 + \gamma_3(q,x)(s-1)^3 + \cdots
\end{align*}
	with
\begin{align*}
	\gamma_0(q,x)& = \sum_{n=1}^\infty\frac{q^{n(1-x)}}{[n]_q} + \frac{(q-1)\log(q-1)}{\log q} - \frac{q-1}{2} + (q-1)(1-x)
\end{align*}
	and
\begin{align*}
	\gamma_1(q,x)& = \Bigg(\sum_{n=1}^\infty\frac{q^{n(1-x)}}{[n]_q} + \frac{(q-1)\log(q-1)}{2\log q} - \frac{q-1}{2} + (q-1)(1-x)\Bigg)\log(q-1)\\
	&\quad + \Bigg(\frac{q-1}{12} - \sum_{n=1}^\infty \frac{(1 + (q^n-1)x)q^{n(1-x)}}{[n]_q(q^n-1)} - \frac{(q-1)(1-x)x}{2}\Bigg)\log q\\
	&\quad\quad + \sum_{n=1}^\infty\frac{q^{n(1-x)}s(n+1,2)}{n![n]_q},
\end{align*}
	where $s(n+1, i)$ are the unsigned Stirling numbers of the first kind.
\end{theorem}

	Further, they addressed the results concerning the special values of $\gamma_0(q,x)$ for $q>1$ and $ 0< x <1$. This is detailed in the following theorem:
	
\begin{theorem} \label{T2}
	Let $q>1$ be any positive algebraic number, $b \geq 3$ be any integer, and $1 \leq a < b/2$ with $(a,b)=1$. Then,
\begin{align*}
	\gamma_0 \left(q, \frac{a}{b} \right) - \gamma_0 \left(q,1- \frac{a}{b} \right) = \left( \frac{q-1}{\log q} \right) \pi\cot \left( \frac{\pi a}{b} \right)+ (2q-3) \left(\frac{1}{2} - \frac{a}{b} \right)
\end{align*}
	is a transcendental number.
\end{theorem}


	Additionally, they presented a conjecture regarding the linear independence of these special values of $\gamma_0(q,x)$, along with the value $1$, when evaluated at rational arguments subject to co-prime conditions. Their conjecture is formulated as follows:
\begin{conjecture}
	Let $q>1$ be any positive algebraic number and $b \geq 3$ be any integer. Then, the following $\varphi(b) +1$ real numbers:
	$$ \left\{1, \gamma_0\left(q, \frac{a}{b}\right): 1 \leq a < b, (a,b)=1 \right\}$$
	are linearly independent over the field of rationals.	
\end{conjecture}
	
	In this article, we consider a number field extension of Conjecture A due to Chatterjee and Garg, specifically applied to a number fields, denoted as $\mathbb{F}$. In this regard, we give the following conjecture:
	
\begin{conjecture}
	Let $q>1$ be any positive algebraic number and $b \geq 3$ be any integer. Let $\mathbb{F}$ be a number field such that $\mathbb{F}$ is linearly disjoint from the cyclotomic field $\mathbb{Q}(\zeta_b)$. Then, the following $\varphi(b) +1$ real numbers:
	$$ \left\{1, \gamma_0\left(q, \frac{a}{b}\right): 1 \leq a < b, (a,b)=1 \right\}$$
	are linearly independent over $\mathbb{F}$.	
\end{conjecture}
	
	Now, we begin our study by defining the following $\mathbb{F}$-linear spaces:
\begin{definition}
	Let $q>1$ be any positive algebraic number and $b \geq 3$ be any integer. Then, $V_\mathbb{F}(q,b)$ is the $\mathbb{F}$-linear space which is defined as: 
	$$V_\mathbb{F}(q,b) = \mathbb{F}-span~ of \left\{1, \gamma_0\left(q, \frac{a}{b}\right): 1 \leq a < b, (a,b)=1 \right\}.$$
\end{definition}

	Then, Conjecture B has an equivalent form which can be stated as follows:
\begin{conjecture}
	Let $q>1$ be any positive algebraic number and $b \geq 3$ be any integer. Let $\mathbb{F}$ be a number field such that $\mathbb{F}$ is linearly disjoint from the cyclotomic field $\mathbb{Q}(\zeta_b)$. Then, 
	$$dim_{\mathbb{F}}V_{\mathbb{F}}(q,b) = \varphi(b) +1.$$
\end{conjecture}

	It is important to observe that the $\mathbb{F}$-dimension of $V_\mathbb{F}(q,b)$, given fixed values of $q$ and $b$, is dependent upon the arithmetic properties of the encompassing number field. Nevertheless, when the number field $\mathbb{F}$ is linearly disjoint from the cyclotomic field $\mathbb{Q}(\zeta_b)$, that is, $\mathbb{F} \cap \mathbb{Q}(\zeta_b) = \mathbb{Q}$, the expectation is that the situation will parallel the notion of linear independence over $\mathbb{Q}$. Thereby, in this direction, we have the following theorems and corollaries:
	
\begin{theorem} \label{T3}
	Let $q>1$ be any positive algebraic number and $b \geq 3$ be any integer. Let $\mathbb{F}$ be a number field such that $\mathbb{F} \cap \mathbb{Q}(\zeta_b) = \mathbb{Q}$. Then, the following set of numbers:
	$$S=\left \{	\gamma_0 \left(q, \frac{a}{b} \right) - \gamma_0 \left(q,1- \frac{a}{b} \right):1 \leq a < \frac{b}{2}, (a,b)=1\right \}$$
	is linearly independent over $\mathbb{F}$. In particular, any non-zero $\mathbb{F}$-linear combination of the elements of $S$ is a transcendental number.
\end{theorem}
	
\begin{corollary}\label{C1}
	Let $q>1$ be any positive algebraic number and $b \geq 3$ be any integer. Let $\mathbb{F}$ be a number field such that $\mathbb{F} \cap \mathbb{Q}(\zeta_b) = \mathbb{Q}$. Then, the following set of numbers:
	$$ \left \{1, \gamma_0 \left(q, \frac{a}{b} \right) - \gamma_0 \left(q,1- \frac{a}{b} \right):1 \leq a < \frac{b}{2}, (a,b)=1 \right \}$$
	is linearly independent over $\mathbb{F}$. 
\end{corollary}
	
\begin{theorem} \label{T4}
	Let $q>1$ be any positive algebraic number and $b \geq 3$ be any integer. Let $\mathbb{F}$ be a number field such that $\mathbb{F} \cap \mathbb{Q}(\zeta_b) = \mathbb{Q}$. Then, at least 
	$$ \frac{\varphi(b)}{2} +1$$
	many numbers of the set $$\left\{1, \gamma_0\left(q, \frac{a}{b}\right): 1 \leq a < b, (a,b)=1 \right\}$$ are linearly independent over $\mathbb{F}$.\\
	Equivalently, $$dim_{\mathbb{F}}V_{\mathbb{F}}(q,b) \geq \frac{\varphi(b)}{2} +1.$$
\end{theorem}

	Now, we look at the case when the number fields intersect $\mathbb{Q}(\zeta_b)$ non-trivially, that is, $\mathbb{F} \cap \mathbb{Q}(\zeta_b) \neq \mathbb{Q}$. In this regard, first we define the following $\mathbb{F}$-linear spaces:
	
\begin{definition}
	Let $q>1$ be any positive rational number and $b \geq 3$ be any integer. Then, $\widehat{V}_\mathbb{F}(q,b)$ is the $\mathbb{F}$-linear space which is defined as: 
	$$\widehat{V}_\mathbb{F}(q,b) = \mathbb{F}-span~ of \left\{1, \gamma_0\left(q, \frac{a}{b}\right): 1 \leq a < b, (a,b)=1 \right\}.$$
\end{definition}
\textbf{Note:} In $V_\mathbb{F}(q,b)$, $q>1$ is any positive algebraic number, whereas in $\widehat{V}_\mathbb{F}(q,b)$, $q>1$ is any positive rational number.

	Now, we present the following theorems concerning $\widehat{V}_\mathbb{F}(q,b)$:
\begin{theorem} \label{T7}
	Let $q>1$ be any positive rational number and $b \geq 3$ be any integer. For $ 1 \leq a < b$, with $(a,b)=1$, let $\kappa_a$ be defined as
	$$\kappa_a = \frac{\log q}{i \pi}\left[\left( \gamma_0 \left(q, \frac{a}{b} \right) - \gamma_0 \left(q,1- \frac{a}{b} \right)\right) - \left( (2q-3) \left(\frac{1}{2} - \frac{a}{b} \right)\right)\right].$$
	If $\kappa_a \in \mathbb{F}$ for some $a$ as above, then
\begin{align*}
	2 \leq dim_{\mathbb{F}}\widehat{V}_{\mathbb{F}}(q,b) \leq \frac{\varphi(b)}{2} +2.
\end{align*}
\end{theorem}

	In light of the preceding theorem, the following outcome can be established:
\begin{corollary}\label{C3}
	Let $q>1$ be any positive rational number and $b \geq 3$ be any integer. Let $\mathbb{F} \cap \mathbb{Q}(\zeta_b) = \mathbb{Q}(\zeta_b)$. Then, we have the following relation:
\begin{align*}
	2 \leq dim_{\mathbb{F}}\widehat{V}_{\mathbb{F}}(q,b) \leq \frac{\varphi(b)}{2} +2.
\end{align*}
\end{corollary}
	
\section{\bf Notations and Preliminaries}
	In this section, we focus on exploring the notations associated with the $q$-series, along with fundamental results that will have a significant impact on the forthcoming sections. \\
	Let $a$ be a complex number. The $q$-analogue of $a$ is expressed by:\\
\begin{align*}
	[a]_q = \frac{q^a - 1}{q - 1}, ~~~~~ q \neq 1.
\end{align*}
	Additionally, the $q$-shifted factorial of $a$ is defined as:
\begin{align*}
	(a;q)_0&=1, ~~~~~ (a;q)_n = \displaystyle\prod_{m=0}^{n-1} (1-aq^n), ~~~~~ n=1,2, \ldots\\
	(a;q)_{\infty}& =\lim_{n\rightarrow\infty}(a;q)_n = \displaystyle\prod_{n=0}^{\infty} (1-aq^n).
\end{align*}
	Furthermore, the $q$-analogue of the Lambert series is represented as:
\begin{align*}
	\mathscr{L}_q(s,x) = \sum_{k=1}^{\infty}\frac{k^s q^{kx}}{1-q^k},~~ s\in \mathbb{C},
\end{align*}
	where $|q| < 1$ and $x>0$. 
%
	
	Next, we examine certain results that play a pivotal role in establishing our results. In 1981, T.\,Okada \cite{TO} established a linear independence result to encompass all derivatives of cotangent values. The formulation of his result is outlined below:
	
\begin{theorem}
	Let $k$ and $q$ be positive integers with $k \geq 1$ and $q > 2$. Let $T$ be a set of $\varphi(q)/2$ representations $\bmod~q$ such that the union $T \cup (-T)$ constitutes a complete set of co-prime residue classes $\bmod~q$. Then, the following set of real numbers: 
\begin{align*}
	\frac{d^{k-1}}{dz^{k-1}} \cot (\pi z)|_{z = a/q},~~~~ a \in T
\end{align*}
	is linearly independent over $\mathbb{Q}$.
\end{theorem}

	Following that, in 2007, a thorough investigation undertaken by Murty and Saradha \cite{MS1} concerning Okada's findings led to the formulation of the following theorem for number fields $\mathbb{F}$:
\begin{theorem} \label{T5}
	Let $k$ and $q$ be positive integers with $k \geq 1$ and $q > 2$. Let $T$ be a set of $\varphi(q)/2$ representations $\bmod~q$ such that the union $T \cup (-T)$ constitutes a complete set of co-prime residue classes $\bmod~q$. Let $\mathbb{F}$ be a number field such that $\mathbb{F} \cap \mathbb{Q}(\zeta_q) = \mathbb{Q}$. Then, the following set of real numbers: 
\begin{align*}
	\frac{d^{k-1}}{dz^{k-1}} \cot (\pi z)|_{z = a/q},~~~~ a \in T
\end{align*}
	is linearly independent over $\mathbb{F}$.
\end{theorem}

	Moreover, in a separate study, Murty and Saradha \cite{MS2} established a result concerning the necessary and sufficient conditions for the convergence of the $L$-function, denoted as $L(k,f)$, for the case $k=1$. This result holds significant importance in the proof of Theorem \ref{T7} and it is outlined as follows:
\begin{theorem}\label{T8}
	Let $f$ be any function defined on the integers and with period $q$. Then,
	$$\sum_{n=1}^{\infty} \frac{f(n)}{n}$$
	converges if and only if
	$$\sum_{a=1}^q f(a)=0,$$
	and in the case of convergence, the value of the series is
	$$-\frac{1}{q} \sum_{a=1}^q f(a) \psi(a / q) .$$
\end{theorem}

	Within the classical framework, Baker's theorem (see ref.\,\cite{AB}) plays a pivotal role in establishing claims related to the logarithms of algebraic numbers. This significance becomes apparent through the following statement:
\begin{theorem} \label{T6}
	If $\alpha_1,\alpha_2,\ldots,\alpha_n$ are non-zero algebraic numbers such that $\log \alpha_1, \ldots, \log \alpha_n$ are linearly independent over the field of rational numbers, then $1, \log \alpha_1, \ldots, \log \alpha_n$ are linearly independent over the field of algebraic numbers.
\end{theorem}
	
	In 2010, Murty and Saradha \cite{MS} (also see \cite{MM}) demonstrated a significant implication of this outcome, presented as follows:
	
\begin{lemma} \label{L1}
	Let $\alpha_1, \ldots, \alpha_n$ be positive algebraic numbers. If $c_0, c_1, \ldots, c_n$ are algebraic numbers with $c_0 \neq 0$, then $$c_0 \pi + \sum_{j=1}^n c_j \log \alpha_j$$ is a transcendental number and hence non-zero.
\end{lemma}
	
	Moreover, in the year 2014, Chatterjee and Murty \cite{CM} documented an alternative version of this lemma, articulated as follows:
	
\begin{lemma} \label{L2}
	Let $\alpha_1, \ldots, \alpha_n$ be positive units in a number field of degree $> 1$. Let $r$ be a positive rational number unequal to $1$. If $c_0, c_1, \ldots, c_n$ are algebraic numbers with $c_0 \neq 0$ and $d$ is an integer, then $$c_0 \pi + \sum_{j=1}^n c_j \log \alpha_j + d \log r$$ is a transcendental number and hence non-zero.
\end{lemma}

	Furthermore, Chatterjee and Garg \cite{TSG3} validated a revised interpretation of Lemma \ref{L1}. The formulation of the outcome is presented as follows:
\begin{lemma} \label{L3}
	Let $\alpha_1, \ldots, \alpha_n$ be positive algebraic numbers such that $\log \alpha_1, \ldots, \log \alpha_n$ are linearly independent over $\mathbb{Q}$. Then, the set of numbers
	$$ \{ 1, \pi, \log \alpha_1, \ldots, \log \alpha_n\} $$ is linearly independent over $\overline{\mathbb{Q}}$. In particular, $\frac{\pi}{\log \alpha_k}$ is transcendental for all $\alpha_k$.
\end{lemma}

	\textbf{Note:} We shall consider $q>1$ throughout this paper.
	
	\section{\bf Proofs of the main theorems}
	
\begin{proof}[\textbf{Proof of Theorem \ref{T3}}]
	For $1 \leq a < b/2$ with $(a,b)=1$, consider $c_a$ $\in \mathbb{F}$, such that:
\begin{align*}
	\sum_{\substack{1 \leq a < b/2 \\ (a , b) = 1}}c_a \left (\gamma_0\left(q,\frac{a}{b}\right) - \gamma_0\left(q,1-\frac{a}{b}\right)\right)  =0.
\end{align*}
	By substituting the value of $\gamma_0\left(q,\frac{a}{b}\right) - \gamma_0\left(q,1-\frac{a}{b}\right)$ from Theorem \ref{T2}, we derive:
\begin{align*}
	&\left(\frac{q-1}{\log q}\right)\pi \left( c_{a_1}\cot \frac{\pi a_1}{b} + \cdots + c_{a_{\varphi(b)/2}}\cot \frac{\pi a_{\varphi(b)/2}}{b}\right)\\
	&+ (2q-3) \left(c_{a_1} \left(\frac{1}{2} - \frac{a_1}{b}\right) + \cdots + c_{a_{\varphi(b)/2}} \left(\frac{1}{2} - \frac{a_{\varphi(b)/2}}{b}\right)\right) = 0,
\end{align*}
	which implies:
\begin{align}
	&\left(\frac{q-1}{\log q}\right)\pi \left( c_{a_1}\cot \frac{\pi a_1}{b} + \cdots + c_{a_{\varphi(b)/2}}\cot \frac{\pi a_{\varphi(b)/2}}{b}\right) \nonumber\\
	&= (3 - 2q) \left(c_{a_1} \left(\frac{1}{2} - \frac{a_1}{b}\right) + \cdots + c_{a_{\varphi(b)/2}} \left(\frac{1}{2} - \frac{a_{\varphi(b)/2}}{b}\right)\right). \label{E1}
\end{align}
	As stated in Lemma \ref{L3}, $\frac{\pi}{\log q}$ is a transcendental number. Consequently, the left-hand side of Equation (\ref{E1}) is an algebraic multiple of a transcendental number, hence it is either zero or transcendental number. On the other hand, the right-hand side of Equation (\ref{E1}) corresponds to an algebraic number. Thus, we can infer:
	$$  c_{a_1}\cot \frac{\pi a_1}{b} + \cdots + c_{a_{\varphi(b)/2}}\cot \frac{\pi a_{\varphi(b)/2}}{b}=0$$
	and 
	$$c_{a_1} \left(\frac{1}{2} - \frac{a_1}{b}\right) + \cdots + c_{a_{\varphi(b)/2}} \left(\frac{1}{2} - \frac{a_{\varphi(b)/2}}{b}\right)=0.$$
	Now, employing Theorem \ref{T5} for the case $k=1$, we ascertain that the following set of numbers:
	$$\left\{\cot \frac{\pi a}{b}: 1 \leq a < b/2, (a,b)=1\right\}$$
	is linearly independent over $\mathbb{F}$. Thus, $c_a = 0$, for all $1 \leq a < b/2$ with $(a,b)=1$. Therefore, the proof of the first part is complete.\\
	In particular, consider $c_a$'s $\in \mathbb{F}$, not all zero, then we have 
\begin{align*}
	\sum_{\substack{1 \leq a < b/2 \\ (a , b) = 1}}c_a \left (\gamma_0\left(q,\frac{a}{b}\right) - \gamma_0\left(q,1-\frac{a}{b}\right)\right) \neq 0.
\end{align*}
	Thus,
\begin{align*}
	\sum_{a}c_a\left(\gamma_0 \left(q, \frac{a}{b} \right) - \gamma_0 \left(q,1 - \frac{a}{b} \right) \right) & =\left( \frac{q-1}{\log q} \right) \pi \sum_{a}c_a \cot \left(\frac{\pi a}{b} \right)\\
	&\quad + (2q-3)  \left(\frac{1}{2}\sum_{a}c_a - \frac{1}{b}\sum_{a}c_a a \right)\\
	& \neq 0
\end{align*}
	Again, from Theorem \ref{T5} and Lemma \ref{L3}, we get the proof of the second part.
\end{proof}
	
\begin{proof}[\textbf{Proof of Corollary \ref{C1}}]
	The corollary is an immediate outcome of both Theorem \ref{T2} and Theorem \ref{T3}.
\end{proof}
	
\begin{proof}[\textbf{Proof of Theorem \ref{T4}}]
	Initially, it is worth noting that the space $V_{\mathbb{F}}(q,b)$ can be spanned by the following sets of numbers:
\begin{align*}
	&\left \{1,	\gamma_0 \left(q, \frac{a}{b} \right) - \gamma_0 \left(q,1- \frac{a}{b} \right): 1 \leq a < \frac{b}{2}, (a,b)=1 \right \},\\
	\text{and}~&\left \{	\gamma_0 \left(q, \frac{a}{b} \right) + \gamma_0 \left(q,1- \frac{a}{b} \right): 1 \leq a < \frac{b}{2}, (a,b)=1\right \}.
\end{align*}
	Now, from Theorem \ref{T2}, we know:
\begin{align*}
	\gamma_0 \left(q, \frac{a}{b} \right) - \gamma_0 \left(q,1- \frac{a}{b} \right) = \left( \frac{q-1}{\log q} \right) \pi \cot \left(  \frac{\pi a}{b} \right) + (2q-3) \left(\frac{1}{2} - \frac{a}{b} \right).
\end{align*}
	and Theorem \ref{T3} establishes the linear independence of the set
	$$\left \{	\gamma_0 \left(q, \frac{a}{b} \right) - \gamma_0 \left(q,1- \frac{a}{b} \right):  1 \leq a < \frac{b}{2}, (a,b)=1\right \}$$
	over $\mathbb{F}$. Finally, using Corollary \ref{C1}, we conclude that
	$$dim_{\mathbb{F}}V_{\mathbb{F}}(q,b) \geq \frac{\varphi(b)}{2} +1.$$
\end{proof}

	Now, prior to delving into the proof of Theorem \ref{T7}, we present the following lemma, which assumes a crucial role in its proof:
\begin{lemma}\label{L4}
	Let $b \geq 3$ be any positive integer and $1 \leq a < b/2$ with $(a,b) = 1$. Then, we have
	$$\cot \left(\frac{\pi a}{b}\right) = \sum_{d=1}^b (\zeta_b^{ad} - \zeta_b^{-ad})B_1\left(\frac{b}{d}\right),$$
	where $B_1(x)$ is the $1$-st Bernoulli polynomial.
\end{lemma}
\begin{proof}[\textbf{Proof}]
	For any periodic function $f$ with period $b$, we have (see \cite{MS1})
	$$2L(k, f)=\frac{(2 \pi i)^k}{k !} \sum_{r=1}^b \widehat{f}(r) B_k(r / b),$$
	where
	$$B_k(x)=\frac{-k !}{(2 \pi i)^k} \sum_{\substack{n=-\infty \\ n \neq 0}}^{\infty} \frac{e^{2 \pi i n x}}{n^k}$$
	is the $k$-th Bernoulli polynomial and
	$$\widehat{f}(n)=\frac{1}{b} \sum_{d=1}^b f(d) e^{2 \pi i d n / b}$$
	is the Fourier transform of $f$. For $a$ co-prime to $b$, we consider the following odd function:
	$$\delta_a(n)= \begin{cases}1, & \text { if } n=a \\ -1, & \text { if } n=b-a \\ 0, & \text { otherwise. }\end{cases}$$
	Then,
	$$2 L\left(k, \delta_a\right)=\frac{(2 \pi i)^k}{k !} \sum_{d=1}^b \widehat{\delta}_a(d) B_k(d / b),$$
	where $\widehat{\delta}_a(n)=\frac{1}{b} \displaystyle\sum_{d=1}^b \delta_a(d) e^{\frac{2 i \pi d n}{b}}=\frac{1}{b}\left[\zeta_b^{a n}-\zeta_b^{-a n}\right]$, where $\zeta_b=e^{\frac{2 \pi i}{b}}$. 
	Hence, we have:
\begin{align}\label{E7}
	2 L\left(k, \delta_a\right)= \frac{(2 \pi i)^k}{b k !} \sum_{d=1}^b \left(\zeta_b^{a d}-\zeta_b^{-a d}\right)B_k(d / b).
\end{align} 
	Also, $L(k,f)$ is related to cotangent function with the following relation \cite{MS1}:
\begin{align}\label{E8}
	L(k,\delta_a) = -\frac{(-1)^k}{(k-1)!b^k}\sum_{\substack{1 \leq a_k < b/2 \\ (a_k , b) = 1}}\delta_a(a_k) \left(\frac{d^{k-1}}{dz^{k-1}}(\pi \cot \pi z)|_{z=a/b}\right).
\end{align}
	Finally, using Theorem \ref{T8}, substitute $k=1$ in Equations (\ref{E7}) and (\ref{E8}). Thus, we obtain:
	$$\cot \frac{\pi a}{b} = i \sum_{d=1}^{b}\left(\zeta_b^{a d}-\zeta_q^{-a d}\right)B_1(d / b).$$
	This completes the proof.	
\end{proof}
	
\begin{proof}[\textbf{Proof of Theorem \ref{T7}}]
	First, it is important to note that the space $\widehat{V}_{\mathbb{F}}(q,b)$ can be spanned by the following sets of real numbers:
\begin{align*}
	&\left \{1,	\gamma_0 \left(q, \frac{a}{b} \right) - \gamma_0 \left(q,1- \frac{a}{b} \right) - \left( (2q-3) \left(\frac{1}{2} - \frac{a}{b} \right)\right): 1 \leq a < \frac{b}{2}, (a,b)=1 \right \},\\
	\text{and}~&\left \{	\gamma_0 \left(q, \frac{a}{b} \right) + \gamma_0 \left(q,1- \frac{a}{b} \right) + \left( (2q-3) \left(\frac{1}{2} - \frac{a}{b} \right)\right): 1 \leq a < \frac{b}{2}, (a,b)=1\right \}.
\end{align*}
	Using Lemma \ref{L4}, we have the following equality: 
	$$\frac{\log q}{i \pi}\left[\left( \gamma_0 \left(q, \frac{a}{b} \right) - \gamma_0 \left(q,1- \frac{a}{b} \right)\right) - \left( (2q-3) \left(\frac{1}{2} - \frac{a}{b} \right)\right)\right]= \sum_{d=1}^{b}\left(\zeta_b^{a d}-\zeta_q^{-a d}\right)B_1(d / b),$$
	where $B_1(x)$ denotes the $1$-st Bernoulli polynomial. Also, it is important to observe that for any positive integer $b \geq 3$ and $1 \leq a < b$ with $(a,b)=1$, we have:
	$$i \cot \left(\frac{\pi a}{b}\right) = \frac{1 + \zeta_b^a}{1 - \zeta_b^a},$$
	belongs to $\mathbb{Q}(\zeta_b)$. Consequently, $\cot(\pi a/b) \in \mathbb{Q}(\zeta_b)$.
	Assuming $\kappa_a \in \mathbb{F}$, we conclude that $$\kappa_a \in F = \mathbb{F} \cap \mathbb{Q}(\zeta_b).$$ Since $F$ is a Galois extension over $\mathbb{Q}$, so each element within the Galois group, $$G = Gal(\mathbb{Q}(\zeta_b)/\mathbb{Q}),$$ when restricted to $F$ gives an automorphism of $F$. Notably, for any $(r,b)=1$, the corresponding element $\sigma_r$ of $G$, induced by the action $\zeta_b \to \zeta_b^r$, maps $\kappa_a$ to $\kappa_{ar}$. Hence, $$\kappa_c \in F,~ \text{for all}~(c,b)=1~ \text{with}~ 1 \leq c < b/2.$$
	Therefore, we deduce that
\begin{align*}
	\left( \gamma_0 \left(q, \frac{a}{b} \right) - \gamma_0 \left(q,1- \frac{a}{b} \right)\right) - \left( (2q-3) \left(\frac{1}{2} - \frac{a}{b} \right)\right) \in i \frac{\pi}{\log q}\mathbb{F},
\end{align*}
	for all $1 \leq a < b/2$ with $(a,b)=1$. Clearly, $\varphi(b)/2$ real numbers 
	$$	\left( \gamma_0 \left(q, \frac{a}{b} \right) - \gamma_0 \left(q,1- \frac{a}{b} \right)\right) - \left( (2q-3) \left(\frac{1}{2} - \frac{a}{b} \right)\right)$$
	are linearly dependent over $\mathbb{F}$. So, these numbers contribute at most $1$ in the dimension. Thus, we have $$dim_{\mathbb{F}}V_{\mathbb{F}}(q,b) \leq \frac{\varphi(b)}{2} +2. $$
	Furthermore, using Lemma \ref{L3}, we get that $1$ and $\frac{\pi}{\log q}$ are linearly independent over $\mathbb{F}$. Thereby, we establish $$dim_{\mathbb{F}}V_{\mathbb{F}}(q,b) \geq 2. $$
	This completes the proof.
\end{proof}

\begin{proof}[\textbf{Proof of Corollary \ref{C3}}]
	As $\mathbb{F} \cap \mathbb{Q}(\zeta_b) = \mathbb{Q}(\zeta_b) \neq \mathbb{Q}$, so the proof of the corollary directly follows form Theorem \ref{T7}. 
\end{proof}

\end{document}